\documentclass[pdflatex,sn-mathphys-num]{sn-jnl}

\usepackage{graphicx}%
\usepackage{multirow}%
\usepackage{amsmath,amssymb,amsfonts}%
\usepackage{amsthm}%
\usepackage{mathrsfs}%
\usepackage[title]{appendix}%
\usepackage{xcolor}%
\usepackage{textcomp}%
\usepackage{manyfoot}%
\usepackage{booktabs}%
\usepackage{algorithm}%
\usepackage{algorithmicx}%
\usepackage{algpseudocode}%
\usepackage{listings}%
\usepackage{caption}

\theoremstyle{thmstyleone}%
\newtheorem{theorem}{Theorem}
\theoremstyle{thmstyletwo}%
\newtheorem{example}{Example}%
\newtheorem{remark}{Remark}%

\theoremstyle{thmstylethree}%
\newtheorem{definition}{Definition}%

\begin{document}

\title[  ]{A particular solution of a higher-order non-homogeneous Cauchy-Euler equation}


\author*[1]{\fnm{Miloud} \sur{Assal}}\email{miloud.assal@fst.rnu.tn}
\equalcont{These authors contributed equally to this work.}
\author[2,3]{\fnm{Skander} \sur{Belhaj}}\email{skander.belhaj@isamm.uma.tn}

\equalcont{These authors contributed equally to this work.}

\affil[1]{\orgdiv{IPEIN}, \orgname{University of Carthage},  \city{Nabeul}, \postcode{ 8000},  \country{Tunisia}}
\affil[1]{\orgdiv{LR18ES45}, \orgname{University of Carthage},  \city{Nabeul}, \postcode{ 8000},  \country{Tunisia}}

\affil[2]{\orgdiv{ISAMM}, \orgname{University of Manouba},  \city{Manouba}, \postcode{2010},  \country{Tunisia}}
\affil[3]{\orgdiv{LAMSIN}, \orgname{University of Tunis El Manar},  \city{Tunis}, \postcode{1002},  \country{Tunisia}}


\abstract{ In this paper   we introduce a new concept of atoms on discrete sets to develop an advanced method to find a particular solution for higher-order non-homogeneous Cauchy-Euler equations.  The proposed method provides also an   approximate  solution by using   approximate roots for the characteristic polynomial of the Cauchy-Euler equation.}




\maketitle

\section{ Introduction}
 The Cauchy-Euler equation known historically as Euler's equation is stated as follows
 \begin{equation}\label{CEE}
  {\sum_{i=0}^n} a_ix^i   y^{(i)}(x) =g(x)
 \end{equation}
where  $g$   is a suitable function, $y^{(i)}$ stands for the $i^{th}$ derivative of the unknown scalar
 function $y(x)$ and  $ a_i, i=1,\ldots,n$ are constant complex numbers.\\
 These differential equations often appear in  algorithms, notably in the analysis of
Quicksort   and engineering applications, among others  when  we solve Laplace's
equation in polar coordinates.  Various approaches have been suggested to solve (\ref{CEE}), see for
 example, \cite{LC} and \cite{HL}. In 2020,  Sangsuwan et al. \cite{24}  proposed the generalized solutions of a certain $n$th order Cauchy-Euler equation by using  a Laplace transform method.\\
Likewise, several papers have appeared in  this field of research,   including   \cite{12} and \cite{10}. Recently Assal and Belhaj \cite{MS} introduced new mathematical tools  to solve a discrete variant  of higher-order non-homogeneous Cauchy-Euler equation.\\
In this work we introduce a new concept of atoms over finite sets of real numbers and we prove new combinatory properties  of  finite sets that we apply further to give an explicit particular solution of a non-homogeneous Cauchy-Euler equation.\\
 {Classical methods for Cauchy–Euler equations typically reduce the problem to constant coefficients via the change of variables $x=e^u$, followed by standard techniques such as undetermined coefficients or variation of parameters. While  our approach rely on supplementary transformations and tend to be more effective for higher-order equations. The proposed atom method works directly in the original variable and provides a systematic alternative that complements classical techniques, especially for more complicated  cases.}\\
This paper is organized as follows;    First  we introduce  new tools to solve equation (\ref{CEE}) where we consider the concept of atoms on finite sets by means of  \cite{MW} which we present in the second section. The third section of this manuscript     covers the main result together with its proof  which is based on atoms properties followed by some   examples.  Section 4 is devoted to  present some  numerical results that illustrate our approach. Finally,  in section 5  we summarize our work with some comments.
\section{Atoms of Discrete Sets}
\begin{definition}  Let    $X=\{x_1, x_2,\ldots, x_n\}$ be any finite set of distinct real numbers. A real-valued function $A$     is called an $X$-atom if it satisfies the following two conditions
\begin{enumerate}
  \item [(i)] $M(A,s):=   \displaystyle \sum_{i=1}^n x_i^s A(x_i) =0$   for $0\leq s\leq n-2$,
  \item [(ii)]$M(A,n-1):=   \displaystyle \sum_{i=1}^n x_i^{n-1} A(x_i) =1$.
\end{enumerate}
  \end{definition}

\begin{remark}
We mention that  the concept of atoms given here is close to that given in \cite{MW}. Here, (i)   is called the cancellation moment condition,  while (ii)  is the size condition of an atom.
\end{remark}
In the following theorem  we introduce an example of  atoms    that will be the main tool to give an explicit particular solution for the non-homogeneous Cauchy-Euler equation.
 \begin{theorem}\label{T1}
Let  $X=\{x_1, x_2,\ldots, x_n\}$  be   any discrete subset of distinct  real numbers   and let   $\left. A : X\longrightarrow  \mathbb{R} \right.$ be the function defined by
$$A(x_i)=  {{ \prod^n_{\substack{ j=1 \\ j\neq i}}}(x_i-x_j)^{-1}}.$$  Then $A$  is an $X$-atom.

\end{theorem}

\begin{proof}
  {To prove Theorem \ref{T1} we introduce the polynomial} $
\omega(x)=\prod_{j=1}^n (x-x_j),$ which is of degree $n$. A direct computation shows that
$
\omega'(x_i)=\prod_{\substack{j=1\\ j\neq i}}^{n}(x_i-x_j),
$
hence
\[
 \label{eq:Aomega}
A(x_i)=\frac{1}{\omega'(x_i)}.
 \]
The Lagrange interpolation basis associated with the nodes
$x_1,\dots,x_n$ is given by
\[
\ell_i(x)=\prod_{\substack{j=1\\ j\neq i}}^{n}
\frac{x-x_j}{x_i-x_j}
=\frac{\omega(x)}{(x-x_i)\omega'(x_i)},
\qquad i=1,\dots,n.
\]
Each polynomial $\ell_i$ has degree $n-1$ and satisfies

\[
 {\ell_i(x)
= A(x_i)\,x^{n-1} + \text{terms of degree }\le n-2.}
\]
Thus, $A(x_i)$ is precisely the leading coefficient of $\ell_i(x)$.
Moreover, for any polynomial $P$ of degree at most $n-1$, the Lagrange interpolation
formula reads
\[
P(x)=\sum_{i=1}^n P(x_i)\,\ell_i(x).
\]

Let $k\in\{0,1,\dots,n-2\}$ and set $P(x)=x^k$.
Then
\[
x^k=\sum_{i=1}^n x_i^k\,\ell_i(x).
\]
Since $\deg P\le n-2$, the coefficient of $x^{n-1}$ in the left-hand side is
zero. Taking the coefficient of $x^{n-1}$ in the right-hand side yields
\[
\sum_{i=1}^n x_i^k\,A(x_i)=0,
\qquad k=0,1,\dots,n-2.
\]
Now consider $P(x)=x^{n-1}$. By interpolation,
\[
x^{n-1}=\sum_{i=1}^n x_i^{n-1}\,\ell_i(x).
\]
Comparing the coefficients of $x^{n-1}$ on both sides gives
\[
 {\sum_{i=1}^n x_i^{n-1}\,A(x_i)=1.}
\]

\end{proof}

\section{Particular Solution   using Atomic Roots}
  A general solution for the non-homogeneous equation of (\ref{CEE}) has   been extensively studied in the literature, so it may be solved
   by the substitution $ x=e^{u}, x>0$, while for  $x<0$, one might replace $x$ by $|x|$. This substitution  may be used to provide a  reduced  linear   equation with constant coefficients.
 \begin{definition}
 \begin{equation}\label{CP} \varphi(r)= {\sum_{i=0}^n} a_ir^{\underline{i}} \end{equation}
is called the characteristic polynomial  of the Cauchy-Euler equation  (\ref{CEE})  where,
 $$r^{\underline{i}}=r(r-1)(r-2)\cdots(r-i+1)    \mbox{ with   }  r^{\underline{0}}=1. $$
\end{definition}
\begin{definition}

For a real number $r$  and for a    function      {\(g \in C^0((0,\infty))\) and \\ \(t^{-r-1}g(t) \in L^1_{\text{loc}}((0,\infty)\))} we define   the operator $I_r(g)$  by
$$I_r(g)(x)=\int_0^x t^{-r }g(t)dt. $$
\end{definition}
\begin{theorem}\label{T2}
If $\varphi$ has $n$ distinct roots $r_1, r_2,\ldots,r_n$, then  a particular solution of (\ref{CEE}) is given by
$$y_p(x)= \sum_{i=1}^nA(r_i)x^{r_i}I_{r_i+1 }(g)(x)$$
where $A$ is the $X$-atom with  $X=\{r_1, r_2,\ldots,r_n\}$  and  {\(g \in C^0((0,\infty))\) and \\ \(t^{-r-1}g(t) \in L^1_{\text{loc}}((0,\infty))\)} for all $\left.r_i, \, i=1,2,\ldots,n.\right.$
  \end{theorem}

\begin{proof}
Computing the derivatives of order $k$, $0\leq k \leq n$ for $y_p$, and taking into account that $M(A,j)=0$ for  all $0\leq j\leq n-2 $  and $M(A,n-1)=1$, we obtain \\\\
$ \begin{array}{ccl}a_1xy_p^{(1)}(x)   & =  &a_1 M(A,0)g(x)   +  a_1\sum_{i=1}^nA(r_i)r_ix^{r_i }I_{r_i+1 }(g)(x)\\\\
 &=&a_1\sum_{i=1}^nA(r_i)r_ix^{r_i }I_{r_i+1 }(g)(x)
 \end{array}
 $  \\\\
and for all  $2\leq k \leq n-1$, we have for some constants $\alpha_{j,k}$ \\\\
$ \begin{array}{ccl}
& &  \qquad a_k x^k y_p^{(k)}(x) =\\\\
       &= & a_k\sum_{i=1}^nA(r_i)\prod^{k-2}_{s=0}(r_i-s )g(x)   +  a_k\sum_{i=1}^nA(r_i)x^{r_i } \prod^{k-1}_{s=0}(r_i-s )I_{r_i+1 }(g)(x)\\\\
        &= & a_k\sum_{j=1}^{k-1}\alpha_{j,k}M(A,j-1) g(x)   +  a_k\sum_{i=1}^nA(r_i) {r_i }^{\underline{k} }x^{r_i } I_{r_i+1 }(g)(x)\\\\
       &= &   a_k\sum_{i=1}^nA(r_i){r_i }^{\underline{k} }x^{r_i }  I_{r_i+1 }(g)(x).
  \end{array}
 $\\
\qquad For $k=n$, we obtain
$$ \begin{array}{ccl}
x^ny_p^{(n)}(x)  \hspace{-0.1cm}&\hspace{-0.1cm}= &\hspace{-0.1cm} M(A,n-1)g(x)+ \sum_{j=1}^{n-1}\alpha_{j,n}M(A,j-1)g(x) \\
 &+&   \sum_{i=1}^nA(r_i){r_i }^{\underline{n} }x^{r_i }  I_{r_i+1 }(g)(x) \\\\
&=& g(x)    +   \sum_{i=1}^nA(r_i)x^{r_i } {r_i }^{\underline{n} }I_{r_i+1 }(g)(x).
\end{array}
$$
Now we  substitute  the above derivatives in equation (\ref{CEE})  to conclude
$$\begin{array}{ccl}
\sum_{i=0}^na_ix^iy^{(i)}&=&  g(x)+  \sum_{i=0}^n A(r_i)x^{r_i}[{r_i}^{\underline{n}}+a_{n-1}{r_i}^{\underline{n\!-\!1}} +..+a_{1}{r_i}^{\underline{1}} +a_0]I_{r_{i+1 }}(g)    \\\\
 &=&  g(x)+  \sum_{i=0}^n A(r_i)x^{r_i}\underbrace{ [\varphi(r_i)]}_{=0}I_{r_{i+1 }}(g)   \\
&=& g(x)
\end{array}
$$
where $\varphi$ is the characteristic polynomial (\ref{CP}) of  equation  (\ref{CEE}).
\end{proof}

Now, we present an application of the main result to different forms of Cauchy-Euler equations.
In the following  examples, we give an exact particular solution for each.

\begin{example}  Consider the equation
\begin{equation}\label{g=poly}
\sum_{i=1}^na_ix^iy^{(i)}(x)=\sum_{j=m}^{N}b_jx^j.
\end{equation}
where  $N\in \mathbb{N}$    with $m=   \max(0, \max([r_i]+1),\,  {0\leq i\leq n} ) $, $(b_j)_{j=m,\ldots,N}$ is a sequence of real numbers   and       $X=\{r_1,\ldots,r_n\} $  is the set of the distinct roots of the  characteristic equation of  (\ref{g=poly}), then
$$y_p(x)=\sum_{i=1}^n\sum_{j=m}^{N} {b_jA(r_i)\over j-r_i}x^{j}.$$
\end{example}

\begin{example} Consider now the non-homogeneous Cauchy-Euler equation\\
 \small $$x^8y^{(8)}+28x^7y^{(7)}+{1007\over 4}x^6y^{(6)}+{3345\over 4}x^5y^{(5)}+{3309\over 4}x^4y^{(4)}-{3\over 2}x^3y^{(3)}$$
 $$+{9\over 2}x^2y^{(2)}-9xy^{(1)}+9y =x^4\ln(x).$$
We can express  the characteristic polynomial as $$\left.\varphi(r)=(r+1)(r-1)(r+2)(r-2)(r+3)(r-3)(r+{1\over 2} )(r-{1\over 2}).\right. $$
We compute  $I_{r_i+1}(x^4\ln(x))$, for $r_i, i=1...8$  we get
\begin{itemize}

 \item $\displaystyle
r_1 = -1,
I_{r_1+1} = \int_0^x t^4 \log(t)\,dt
= \frac{x^5}{5}\left(\log(x)-\frac{1}{5}\right)
 $

\item $\displaystyle
r_2 = 1, \quad
I_{r_2+1} = \int_0^x t^2 \log(t)\,dt
= \frac{x^3}{3}\left(\log(x)-\frac{1}{3}\right)
 $

\item $\displaystyle
r_3 = -2,\quad
I_{r_3+1} = \int_0^x t^5 \log(t)\,dt
= \frac{x^6}{6}\left(\log(x)-\frac{1}{6}\right)
 $
\item $\displaystyle
r_4 = 2, \quad
I_{r_4+1} = \int_0^x t \log(t)\,dt
= \frac{x^2}{2}\left(\log(x)-\frac{1}{2}\right)
 $

\item $\displaystyle
r_5 = -3,\quad
I_{r_5+1} = \int_0^x t^6 \log(t)\,dt
= \frac{x^7}{7}\left(\log(x)-\frac{1}{7}\right)
 $

\item $\displaystyle
r_6 = 3, \quad
I_{r_6+1} = \int_0^x \log(t)\,dt
= x\left(\log(x)-1\right)
$
\item $\displaystyle
r_7 = -\tfrac{1}{2},  \quad
I_{r_7+1} = \int_0^x t^{7/2} \log(t)\,dt
= \frac{2x^{9/2}}{9}\left(\log(x)-\frac{2}{9}\right)
$
  \item $\displaystyle
r_8 = \tfrac{1}{2}, \quad
I_{r_8+1} = \int_0^x t^{5/2} \log(t)\,dt
= \frac{2x^{7/2}}{7}\left(\log(x)-\frac{2}{7}\right)
 $
\end{itemize}
 According to Theorem \textbf{\ref{T2}}, a particular solution for the above equation is given as follows
 $$y_p(x)= \sum_{i=1}^8A(r_i)x^{r_i}I_{r_i+1 }(x^4\ln(x)) = {1\over 19845}x^4ln(x) - {898\over 6251175}x^4.$$
 A simple computation shows that $y_p$  is a particular solution of  the given equation above.
\end{example}

 \begin{example} \label{example0}
		Consider the non-homogeneous Cauchy-Euler equation $$x^5y^{(5)}+{17\over 2}x^4y^{(4)}+{19\over 2}x^3y^{(3)}-{21\over 2}x^2y^{(2)}+3xy^{(1)}-3y=x^8\sin(x).$$
				 We can express  the characteristic polynomial as $$\varphi(r)=(r-1)(r+1)(r-2)(r+3) (r-{1\over 2}).   $$
We compute  $I_{r_i+1}(x^8\sin(x))$, for $r_i, i=1...5$  we get
\begin{itemize}
    \item $\displaystyle
r_1 = -1,\quad
I_{r_1+1}
= \int_0^x t^8 \sin(t)\,dt
= 40320- \sin(x)\bigl(-8x^7 + 336x^5 - 6720x^3 + 40320x\bigr)
- \cos(x)\bigl(x^8 - 56x^6 + 1680x^4 - 20160x^2 + 40320\bigr)
$

 \item $\displaystyle
r_2 = 1,\quad
I_{r_2+1}
 = \int_0^x t^6 \sin(t)\,dt
 = \sin(x)\bigl(6x^5 - 120x^3 + 720x\bigr)
- \cos(x)\bigl(x^6 - 30x^4 + 360x^2 - 720\bigr) - 720
$

 \item $\displaystyle
r_3 = 2,\quad
I_{r_3+1}
= \int_0^x t^5 \sin(t)\,dt
= \sin(x)\bigl(5x^4 - 60x^2 + 120\bigr)
- \cos(x)\bigl(x^5 - 20x^3 + 120x\bigr)
 $

 \item $\displaystyle
r_4 = -3,\quad
I_{r_4+1}
= \int_0^x t^{10} \sin(t)\,dt
= \sin(x)\bigl(10x^9 - 720x^7 + 30240x^5 - 604800x^3 + 3628800x\bigr)
- \cos(x)\bigl(
x^{10} - 90x^8 + 5040x^6 - 151200x^4
+ 1814400x^2 - 3628800
\bigr)
 - 3628800
$
\item $\displaystyle
r_5 = \tfrac12,\quad
I_{r_5+1}
= \int_0^x t^{13/2} \sin(t)\,dt
= \frac{135135}{64}x^{1/2}\cos(x)
- \frac{9009}{16}x^{5/2}\cos(x)
+ \frac{143}{4}x^{9/2}\cos(x)
- x^{13/2}\cos(x)
+ \frac{45045}{32}x^{3/2}\sin(x)
- \frac{1287}{8}x^{7/2}\sin(x)
+ \frac{13}{2}x^{11/2}\sin(x)
- \frac{135135\sqrt{2\pi}}{128}
\operatorname{fresnelc}\!\left(
\sqrt{\frac{2x}{\pi}}
\right)
$
\end{itemize}

According to Theorem \textbf{\ref{T2}}, a particular solution for the above equation is given as follows


\begin{center}
$ y_p(x)= \sum_{i=1}^5A(r_i)x^{r_i}I_{r_i+1 }(x^8\sin(x))=180x- {2240\over x}+ {25920 \over x^3}-10720 {cos(x)\over x} - x^{3}cos(x)+ {25920cos(x) \over x^3} + {1265\over 4}xcos(x) - 2080sin(x) + {51 \over 2}x^{2}sin(x) + 25920{sin(x)\over x^2} - {2145\over 8}\sqrt{2\pi x}     C(\sqrt{2x\over \pi}).$
\end{center}


where $C(x)=\int_0^x \cos(t^2)dt$  denotes  the  Fresnel cosine integral.\\
 A simple computation shows that $y_p$  is a particular solution of  the given equation above.
    \end{example}

\section{Approximate Solution}
It is often hard to find the exact roots of any given  polynomial; however there still exist different ways to find approximate values of these roots. In this section we prove that  approximate values of roots for the characteristic polynomial of equation (\ref{CEE}) lead to an approximate particular solution with an accurate error rate and we illustrate these results by a numerical approach.

Now  we present a brief analysis of the accuracy of $y(x)$ computed by Theorem \ref{T2}.\\
\begin{theorem}\label{Approximatethm} Let \textcolor[rgb]{1.00,0.00,0.00}{$\varphi$ be} the characteristic polynomial of Cauchy-Euler equation (\ref{CEE}) with $n$ distinct roots $r_1,\ldots,r_n$. If $r_{i,\varepsilon} =r_i+\varepsilon$ is an approximate value of $r_i$, $i=1,\ldots,n$ where $\varepsilon\in\left( -\mu;\mu\right)$ and $\mu>0$, and $|g(x)|\leq C|x|^\alpha$ for some $\alpha > max_{1\leq i\leq n}(r_i)$ then
$$y_{p,\varepsilon}(x)= \sum_{i=1}^nA(r_i)x^{r_{{i,\varepsilon}}}I_{r_{i,\varepsilon}+1 }(g)(x).$$
Moreover, $y_{p,\varepsilon}(x)$ converges uniformly on any compact $K$ of $\mathbb{R}^+$  (resp. $\mathbb{R}^-$ ) to $y_{p}(x)$ and
  $$Error_{\varepsilon}(x)=| y_p(x)-y_{p,\varepsilon}(x)|
   \leq C_K{\varepsilon \over \beta-\varepsilon}, $$
  where   $\beta=  \min_{1\leq i\leq n}(\alpha-r_i)$  and $C_K$  is a constant which depends  on $K$.
\end{theorem}
\begin{proof}
Taking into account that   $A(r_{i,\varepsilon})=A(r_i)$ it follows for $x>0$  and $\varepsilon$ small enough  \\\\
  $ \begin{array}{ccl}
      | y_p(x)-y_{p,\varepsilon}(x)| & =&\displaystyle \Big{|}\sum_{i=1}^nA(r_i)x^{r_i} \Big{[}   \int_0^x{g(t)\over t^{r_i +1}}dt - x^{ \varepsilon }\int_0^x{g(t)\over t^{r_i+\varepsilon +1}}dt\Big{]}\Big{|} \\\\
      & \leq &\displaystyle \sum_{i=1}^n|A(r_i)|x^{r_i}{\varepsilon x^\alpha\over (\alpha-r_i)(\alpha-r_i-\varepsilon)} \leq \displaystyle C_K{\varepsilon \over \beta-\varepsilon}.
    \end{array}
$ \\\\
 Likewise,  we get the result for $x<0$.
 \end{proof}
\vspace{1cm}
To illustrate the results obtained in  Theorem \ref{Approximatethm}, we plot the errors in Figure \ref{fig:0} as follows: $\left.Error_{\varepsilon_{j}}(x)=\left| y_p(x)-y_{p,\varepsilon_{j}}(x)\right|\right.$ of Example \ref{example0}  with  $\varepsilon_{j}=c\times10^{-j}$ for $j=1,\ldots,5$ and $c$ is randomly taken  in $\left( -1,1\right).$
\begin{figure}[H]
\captionsetup{justification=centering}
\centering
\includegraphics[height=2in,width=4.5in]{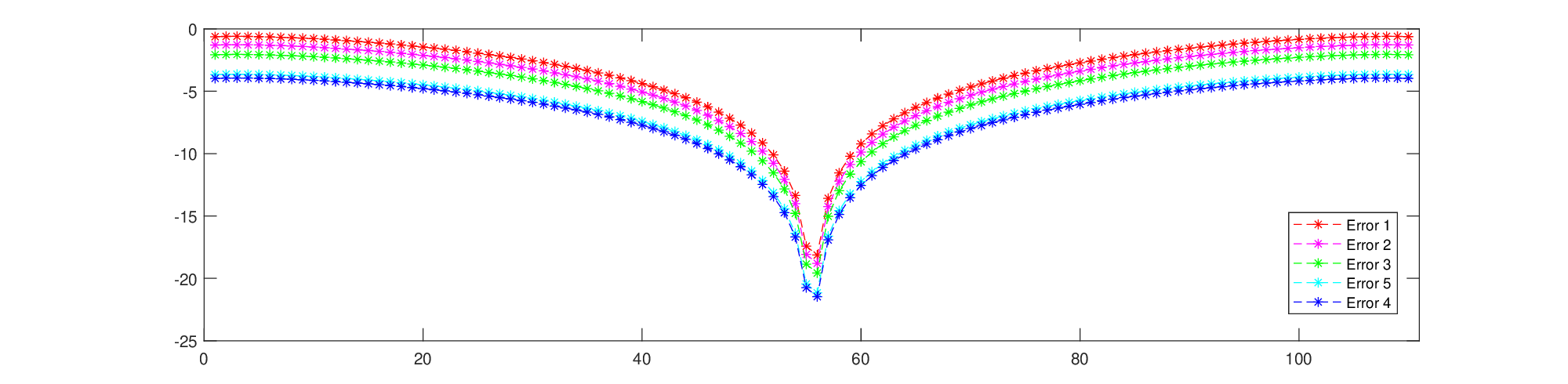}
\caption{Error 1: $Error_{\varepsilon_{1}}(x)$ (red), Error 2: $Error_{\varepsilon_{2}}(x)$ (magenta), Error 3: $Error_{\varepsilon_{3}}(x)$ (green), Error 4: $Error_{\varepsilon_{4}}(x)$ (cyan) and Error 5: $Error_{\varepsilon_{5}}(x)$ (blue) for Example \ref{example0}. }
\label{fig:0}
\end{figure}
Then, the behavior of error  shows that numerical results are very close to the theoretical one.
\subsection{Numerical results}
In what follows,  we consider a family of approximate roots to get an approximate solution for the exact one with a small error. Thus, we highlight the potential benefits of our method in terms of numerical stability. So, we are interested in applying Theorem \ref{T2} to find  numerically a particular solution for Cauchy-Euler equations of the following three tests:\\
\textbf{Test $1$. } $g(x)=x^{4}\sin x$ odd\\
\textbf{Test $2$. } $g(x)=x^{5}\sin x$ even\\
\textbf{Test $3$. } $\left.g(x)=x^{6}(\sin x+\cos x) \text{ any.}\right.$\\
Numerical tests were implemented in Matlab 9.4.0.813654 (R2018a) and run on an Intel(R) Core(TM) i7-8550U CPU @ 1.80GHz laptop with 8GB of RAM and 1.99GHz processor.
\subsection{Accuracy}
We randomly  perturb five times the roots $1,$ $-2,$ $3,$ $1/2$ of the characteristic polynomial\\ $$\left.\varphi(r)=r^4 - {5r^3\over 2} - 4r^2 + {17r\over 2} - 3\right.$$ for Tests 1-3 by adding $\varepsilon _{j}=c\times10^{-j}$ to all the roots for $j=1,\ldots,5$ and $c$ is randomly taken  in $\left( -1,1\right)$.\\
To illustrate the error analysis results established in Theorem~4.1, Figures~4.2 to~4.4 display the exact solution \( y(x) \) along with the perturbed solutions \( y_{\varepsilon_j}(x) \), for \( j = 1, \ldots, 5 \), corresponding to Tests~1 through~3. In each figure, the exact solution \( y(x) \) is plotted as a solid (continuous) blue line, while the approximate solutions \( y_{\varepsilon_1}(x) \), \( y_{\varepsilon_2}(x) \), \( y_{\varepsilon_3}(x) \), \( y_{\varepsilon_4}(x) \), and \( y_{\varepsilon_5}(x) \) are represented as dashed lines in magenta, red, green, yellow, and cyan, respectively. Specifically, Figure~4.2 corresponds to Test~1, Figure~4.3 to Test~2, and Figure~4.4 to Test~3.

\begin{figure}[H]
    \centering
  \includegraphics[scale=0.4]{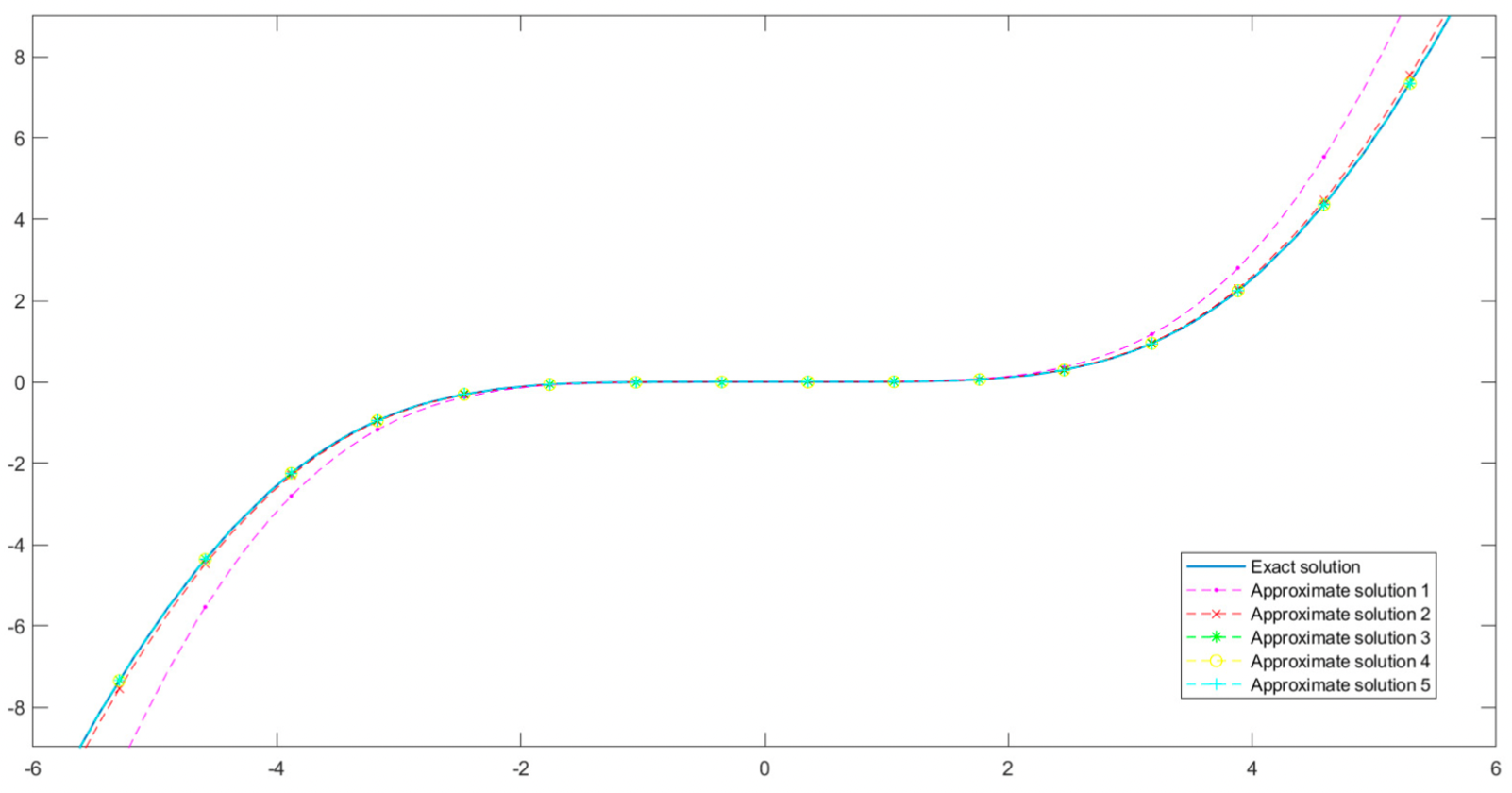}
    \caption{Test $1$. $g(x)=x^{4}\sin x$ odd}
   \label{fig:1}
    \centering
 \includegraphics[scale=0.65]{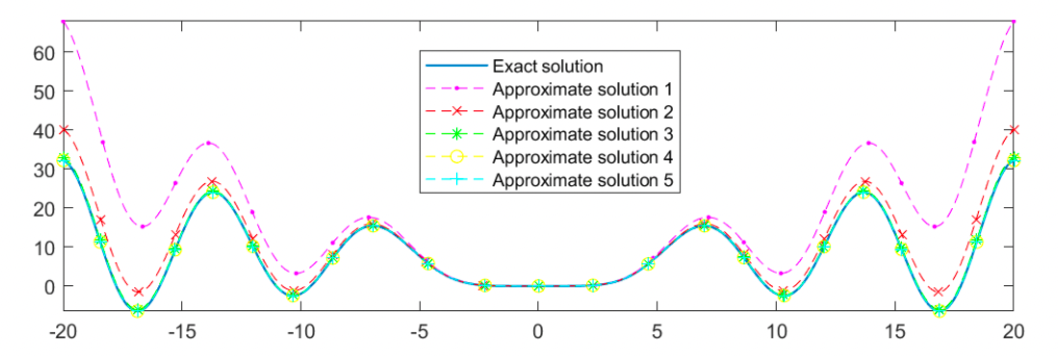}
    \caption{Test $2$. $g(x)=x^{5}\sin x$ even}
    \label{fig:2}
    \centering
 \includegraphics[scale=0.29]{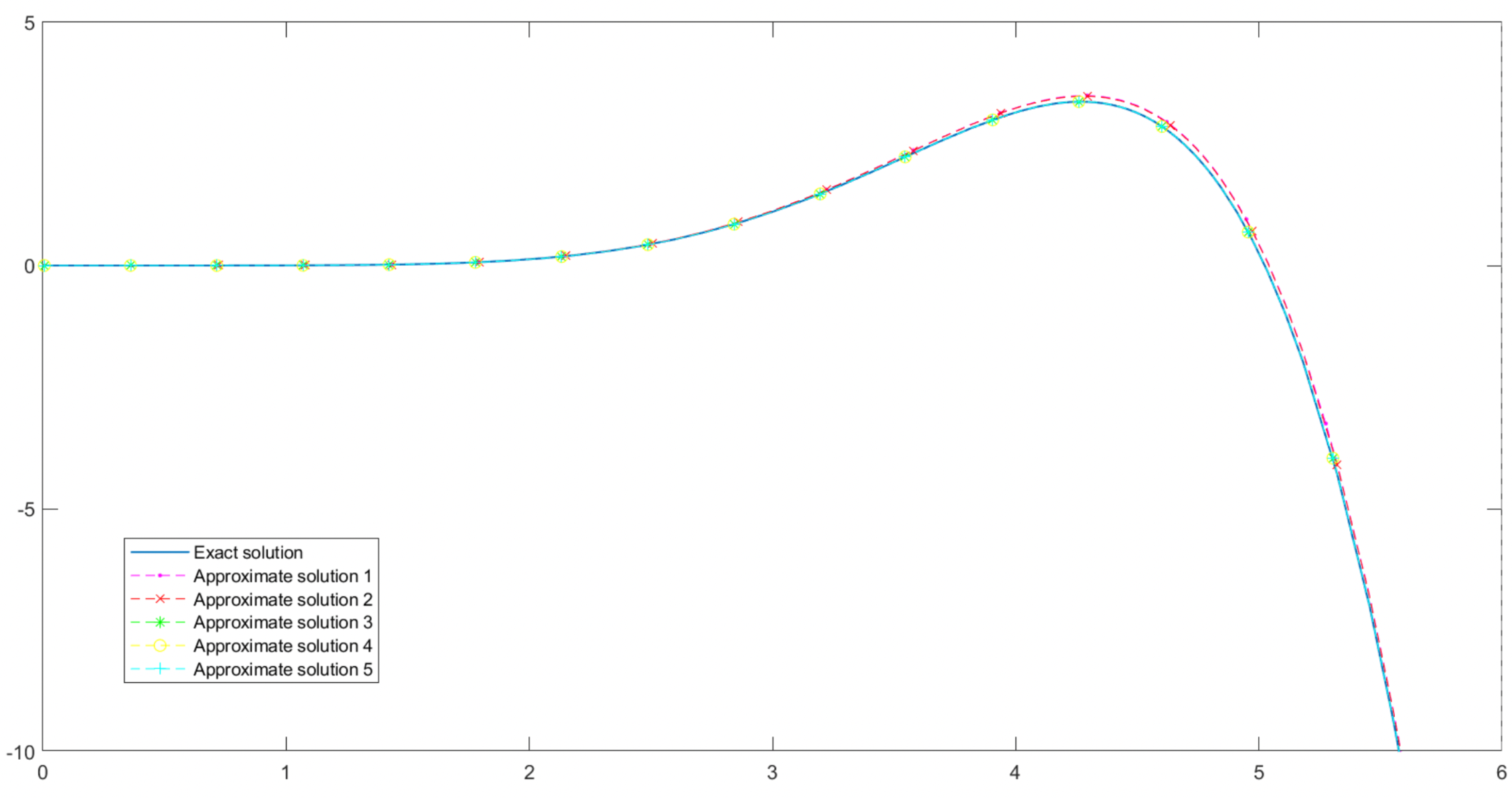}
    \caption{Test $3$. $\left.g(x)=x^{6}(\sin x+\cos x) \text{ any.}\right.$}
    \label{fig:3}
\end{figure}

Figures \ref{fig:1}-\ref{fig:3} show that our method is very accurate with respect to rounding errors and the choice of $g(x)$.\\
A  wider  numerical stability for tests confirms that our approach is quite stable and the results do not diverge for the considered perturbation above  when the number of roots $n$ increases.

\subsection{Timings}
We have checked the growth rate of the running time of our approach on $n$ randomly generated roots ranging between
2 and 50 and $g(x)$ is defined as in  Test 2. Figure \ref{fig:2} shows the running time against the degree in $log$-$log$ scale, with a linear fit and its equation. Roughly speaking, the running time grows as $O(n^\alpha)$, where $\alpha$ is the coefficient of the linear term in the equation, i.e., 1.06 in our case.
\begin{figure}[H]
    \centering
    \captionsetup{justification=centering}
 \includegraphics[scale=0.41]{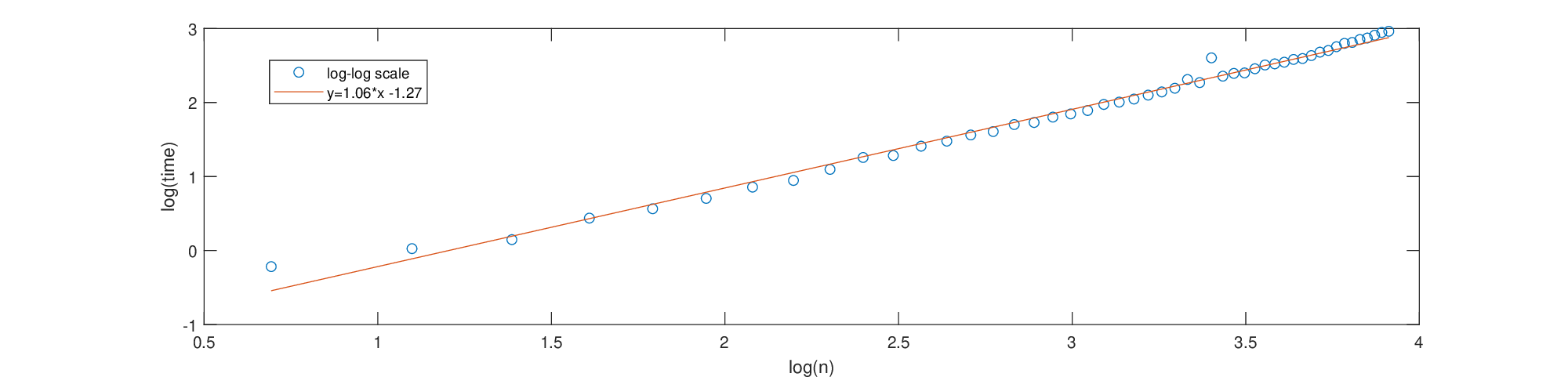}
     \caption{Running time (in seconds)}
    \label{fig:4}
\end{figure}

\section{Concluding remarks}
Cauchy-Euler equations  naturally occur in a large variety of applications including the studies of algorithms, Quicksort, engineering applications and more. The main problem   is to find a general solution for its non-homogeneous form, which  is a key topic in this area, where few results have been    provided   for a restricted particular  equation.
This paper provides a  contribution   to the development of  analytic and numerical methods for solving such an equation. We have introduced a new concept of atoms  defined on   finite subsets of  real numbers  which become   crucial tools to express an analytical particular   solution  for its non-homogeneous form  generated by the set of roots of its characteristic polynomial. This is not always possible to calculate in general. So a numerical procedure must be used.  The proposed  method focuses  on a random   perturbation of  known distinct   roots in order  to get an approximate  solution.\\
This method is able to provide a particular solution  for a given higher-order non-homogeneous Cauchy-Euler equation. The numerical solution is compared with the exact solution and                the results   presented indicate that our method is a powerful mathematical tool to give a particular solution to the non-homogeneous Cauchy-Euler equation. In particular, a numerical  analysis of the stability applied to our tests confirms that our approach is stable and that the above perturbation does not affect the convergence when the number of roots $n$ increases.


\begin{thebibliography}{10}
\bibitem{MS} Assal  Miloud  and Belhaj Skander, A Particular Solution for Higher Order Non-homogeneous Discrete
Cauchy-Euler Equations,  Difference Equations, Discrete Dynamical Systems and Applications, DOI.:	 10.1007/978-3-031-51049-6-19, (2024).
\bibitem {MW} M. Assal, A. Rahmouni, An Improved Hardy's Inequality, Acta Mathematica Scientia, Vol. 33, N. 5, p. 1382-1386 (2013).
 \bibitem {HL} H.-H. Chern, H.-K. Hwang, T.-H. Tsai, An asymptotic theory for Cauchy-Euler differential
equations with applications to the analysis of algorithms, J. Of Algorithms, Vol. 44, p. 177-225 (2002).
\bibitem{12} S. Hinpang, On solution of the second order Euler equation in the distribution space,
Chiang Mai University: Graduate School, Chiang Mai (2001).
\bibitem{10} N. Hongsit, On lattice of the fourth order Euler equation, Chiang Mai University:
Graduate School, Chiang Mai (2001).
\bibitem {LC}L. Qiusheng, C. Hong, L. Guiqing, Static and dynamic analysis of straight bars with variable crosssection, Computer $\&$ Structures Vol. 59, No. 6, p. 1185-1191 (1996).
\bibitem{24} A. Sangsuwan, K. Nonlaopon and S. Orankitjaroen, The generalized solutions of a certain nth order Cauchy-Euler equation, Asian-European Journal of Mathematics, Vol. 13, No. 1 (2020).



\end{thebibliography}
\end{document}